\documentclass{amsart}
\usepackage{hyperref}

\newtheorem{theorem}{Theorem}[section]
\newtheorem{lemma}[theorem]{Lemma}
\newtheorem{corollary}[theorem]{Corollary}

\theoremstyle{definition}

\theoremstyle{definition} 
\newtheorem{remark}[theorem]{Remark}
\newtheorem*{remark*}{Remark}

\newcommand{\R}{\mathbb{R}}

\newcommand{\E}{\operatorname{\mathsf{E}}}
\renewcommand{\P}{\operatorname{\mathsf{P}}}

\newcommand{\dd}{\operatorname{d}}
\renewcommand{\dd}{\mathrm{d}}

\newcommand\tsum{\textstyle\sum\nolimits}

\newcommand{\al}{\alpha}
\newcommand{\ga}{\gamma}
\newcommand{\vp}{\varepsilon}
\newcommand{\vpi}{\varphi}

\numberwithin{equation}{section}

\begin{document}

\title{Explicit additive decomposition of norms on $\R^2$}
\markright{Decomposition of norms}
\author{Iosif Pinelis}

\address{Department of Mathematical Sciences, Michigan Technological University, Houghton, Michigan, 49931}
\email{ipinelis@mtu.edu}



%
%
%


\date{\today}


\begin{abstract}
A well-known result by Lindenstrauss is that any two-dimensional normed space can be isometrically imbedded into $L_1(0,1)$. 
We provide an explicit form of a such an imbedding. The proof is elementary and self-contained. 
Applications are given concerning the following: (i) explicit representations of the moments of the norm of a random vector $X$ in terms of the characteristic function and the Fourier--Laplace transform of the distribution of $X$; (ii) an explicit and partially improved form of the exact version of the Littlewood--Khinchin--Kahane inequality obtained by Lata{\l}a and Oleszkiewicz; (iii) an extension of an inequality by Buja--Logan--Reeds--Shepp,  arising from a statistical problem.   
\end{abstract}

\maketitle

%
%

\section{Main result and discussion}
Let $V$ be a vector space over $\R$ endowed with a norm $\|\cdot\|$, with the dual space $V^*$. 
For any $x\in V$ and $\ell\in V^*$, let $x\ell$ denote the value of the linear functional $\ell$ at $x$. 
Let us say that the norm $\|\cdot\|$ admits 
an additive decomposition if there exists a Borel measure $\mu$ on $V^*$ such that  
\begin{equation}\label{eq:general V}
	\|x\|=\int_{V^*}|x\ell|\,\mu(\dd\ell)  \quad\text{for all $x\in V$.}
\end{equation} 
Clearly, such a decomposition exists if $V$ is one-dimensional. In this note, an explicit decomposition of the form \eqref{eq:general V} will be given in the case when $V$ is two-dimensional. It is well known that, in general, there is no such decomposition if the dimension of $V$ is greater than $2$; cf.\ Remark~\ref{rem:} in the present note. 

To state our main result, let us recall some basic facts about convex functions. 
Suppose that a function $f\colon\R\to\R$ is convex. Then $f$ is continuous and has finite nondecreasing right and left derivatives $f'_+$ and $f'_-$, which are right- and left-continuous, respectively. 
Moreover, the function $f':=(f'_++f'_-)/2$ is nondecreasing as well. The Lebesgue--Stieltjes integral $\int_\R\vpi(t)\,\dd f'(t)$ is the Lebesgue integral $\int_\R\vpi\,\dd\nu$, where $\nu$ is the Borel measure determined by the condition that $\nu\big((a,b]\big)=f'_+(b)-f'_+(a)$ for all real $a$ and $b$ such that $a<b$. The latter condition is equivalent to each of the following conditions: (i) $\nu\big([a,b)\big)=f'_-(b)-f'_-(a)$ for all real $a$ and $b$ such that $a<b$ and (ii) $\nu\big([a,b)\big)+\nu\big((a,b]\big)=2\big(f'(b)-f'(a)\big)$ for all real $a$ and $b$ such that $a<b$. 

Now we are ready to state the following explicit additive decomposition of an arbitrary norm in the case when $V=\R^2$. 

\begin{theorem}\label{th:}
Let $\|\cdot\|$ be any norm on $\R^2$. 
Let $N(u):=\|(u,1)\|$ for all real $u$. Then the function $N$ is convex, the limit 
\begin{equation}\label{eq:c}
	c:=\lim_{u\to\infty}\big(\|(u,1)\|-2\|(u,0)\|+\|(u,-1)\|\big)
\end{equation}
exists and is finite and nonnegative, and  
\begin{equation}\label{eq:N=}
	\|(u,v)\|=\frac{c|v|}2+\frac12\,\int_\R|u-tv|\,\dd N'(t)\quad\text{for all $(u,v)\in\R^2$.} 
\end{equation}
\end{theorem}

\noindent Obviously, \eqref{eq:N=} is an explicit additive decomposition of the form \eqref{eq:general V}, with $\mu$ defined by the condition that $2\int_{\R^2}g\big((s,t)\big)\,\mu(\dd s\times\dd t)=cg\big((0,1)\big)+\int_\R g\big((1,-t)\big)\,\dd N'(t)$ for all nonnegative Borel-measurable functions $g\colon\R^2\to\R$.  

Special cases of representation \eqref{eq:N=}, for the $\ell_p$ norm on $\R^2$ with $p>1$, 
are 
\begin{align*}
	&\|(u,v)\|_p=(|u|^p+|v|^p)^{1/p}=\frac{p-1}2\,\int_\R|u-tv|\,|t|^{p-2}\,(|t|^p+1)^{1/p-2}\,\dd t 
\intertext{when $p\in(1,\infty)$ and}	
	&\|(u,v)\|_\infty=\max(|u|,|v|)=\tfrac12\,(|u+v|+|u-v|) 
\end{align*}
for all $(u,v)\in\R^2$. 
In these cases, the constant $c$ in \eqref{eq:N=}--\eqref{eq:c} is $0$. 
A simple case with a nonzero $c$ is given by the formula $\|(u,v)\|=\|(u,v)\|_1
=|u|+|v|$ for $(u,v)\in\R^2$, with $c=2$. 


The proof of Theorem~\ref{th:} relies on 

\begin{lemma}\label{lem:} 
Suppose that $f\colon\R\to\R$ is a convex function such that for some real $k$ there exist finite limits 
\begin{equation}\label{eq:d+-}
	d_+:=d_{f,k;+}:=\lim_{u\to\infty}[f(u)-ku]\quad\text{and}\quad d_-:=d_{f,k;-}:=\lim_{u\to-\infty}[f(u)+ku]. 
\end{equation}
Then for all $u\in\R$ 
\begin{equation}\label{eq:f=}
 f(u)=\frac{d_++d_-}2+\frac12\,\int_\R|u-t|\,\dd f'(t).
\end{equation}
\end{lemma}

\begin{proof}[Proof of Lemma~\ref{lem:}]
Since the function $f'$ is nondecreasing, there exist limits $k_\pm:=\lim_{x\to\pm\infty}f'(x)\in[-\infty,\infty]$. Moreover, for any real $u>0$ one has $f(u)-ku=f(0)+\int_0^u(f'(t)-k)\,\dd t$, which converges to a finite limit (as $u\to\infty$) only if $k_+=k$. Similarly, $k_-=-k$. So, in view of \eqref{eq:d+-}, for any real $u$ 
\begin{align*}
	f(u)+ku
	=d_-+\int_{-\infty}^u(f'(z)+k)\,\dd z
	=&d_-+\int_{-\infty}^u dz\int_{-\infty}^z\,\dd f'(t) \\ 
		=&d_-+\int_{-\infty}^u df'(t)\int_t^u \dd z\, \\ 
				=&d_-+\int_\R \max(0,u-t)\,\dd f'(t),   
\end{align*} 
so that $$f(u)+ku=d_-+\int_\R\max(0,u-t)\,\dd f'(t).$$ 
Similarly, 
$f(u)-ku=d_+ 
+\int_\R \max(0,t-u)\,\dd f'(t)$ for any real $u$. 
Adding the last two identities, one obtains \eqref{eq:f=}. 
\end{proof}

\begin{proof}[Proof of Theorem~\ref{th:}]
That the function $N$ is convex follows immediately from the convexity of the norm. 
Note next that the limits $d_{f,k;\pm}$ in \eqref{eq:d+-} exist in $[-\infty,\infty]$ for any convex function $f\colon\R\to\R$ and any real $k$. 
On the other hand, for all real $u$ 
\begin{equation*}
	\big|N(u)-|u|\,\|(1,0)\|\big|=\big|\|(u,1)\|-\|(u,0)\|\big|\le\|(0,1)\|,  
\end{equation*}
by the norm inequality. 
So, the limits $d_\pm=d_{f,k;\pm}$ in \eqref{eq:d+-} exist and are finite for $f=N$ and 
\begin{equation}\label{eq:k=}
	k=\|(1,0)\|. 
\end{equation}
Therefore, by Lemma~\ref{lem:}, \eqref{eq:f=} holds with $f=N$ and $d_\pm=d_{N,\|(1,0)\|;\pm}$. 
It follows that, with these $d_\pm$, 
\begin{equation}\label{eq:2N=}
	2\|(u,v)\|=2|v|\,\|(u/v,1)\|=2|v|\,N(u/v)=(d_+ + d_-)|v| + \int_\R|u-tv|\,\dd N'(t)
\end{equation}
for all real $u$ and all real $v\ne0$. The last expression in \eqref{eq:2N=} is continuous in $v\in\R$ by dominated convergence -- because, by \eqref{eq:2N=} with $(u,v)=(0,1)$, one has $\int_\R|t|\,\dd N'(t)=2\|(0,1)\|-(d_++d_-)<\infty$. 
Thus, one has \eqref{eq:N=} -- with $d_+ + d_-$ in place of $c$ -- for all $(u,v)\in\R^2$. 

Moreover, 
\begin{align*}
	d_++d_-=&\lim_{u\to\infty}\big(N(u)-ku+N(-u)-ku\big) \\ 
	=&\lim_{u\to\infty}\big(\|(u,1)\|-2\|(1,0)\|u+\|(-u,1)\|\big) \\ 
		=&\lim_{u\to\infty}\big(\|(u,1)\|-2\|(u,0)\|+\|(u,-1)\|\big)=c\ge0, 
\end{align*}
by \eqref{eq:c} and, again, the convexity of the norm. 

This completes the proof of Theorem~\ref{th:}. 
\end{proof}

From the proofs of Theorem~\ref{th:} and Lemma~\ref{lem:}, it follows that the nondecreasing function $N'$ tends to $\pm k$ as $x\to\pm\infty$, where $k$ is as in \eqref{eq:k=}. So, 
\begin{equation*}
	F:=\tfrac12+\tfrac1{2k}\,N'
\end{equation*}
is a cumulative probability distribution function (cdf), regularized in the sense that $2F(u)=F(u+)+F(u-)$ for all real real $u$. Let the function $F^{-1}\colon(0,1)\to\R$ be (the smallest, left-continuous generalized inverse to $F$,) defined by the condition 
\begin{equation*}
	F^{-1}(s)=\inf\{u\in\R\colon F(u)\ge s\}\quad\text{for }s\in(0,1). 
\end{equation*}
A well-known fact is that, if $S$ is a random variable (r.v.) uniformly distributed on the interval $(0,1)$, then the regularized cdf of the r.v.\ $F^{-1}(S)$ is $F$.  
So, \eqref{eq:N=} can be rewritten as 
\begin{equation}\label{eq:imbed}
	\|(u,v)\|=\int_0^1|u\,\xi(s)+v\,\eta(s)|\,\dd s   
\end{equation}
for all $(u,v)\in\R^2$, 
where 
\begin{equation}\label{eq:xi,eta}
	\xi(s):=\left\{
	\begin{aligned}
	1&\text{ if }0<s<\tfrac12, \\ 
	0&\text{ if }\tfrac12\le s<1, \\ 
	\end{aligned}
	\right.
	\qquad
	\eta(s):=\left\{
	\begin{alignedat}{2}
	-F^{-1}(2s)&&&\text{ if }0<s<\tfrac12, \\ 
	c&&&\text{ if }\tfrac12\le s<1.  
	\end{alignedat}
	\right.
\end{equation}
Thus, the mapping $(u,v)\mapsto u\,\xi+v\,\eta$ is a linear isometric imbedding of $\R^2$ \big(endowed with the arbitrary norm $\|\cdot\|$\big) into $L_1(0,1)$. 

That any two-dimensional normed space is isometric to a subspace of $L_1(0,1)$ was shown by Lindenstrauss \cite[Corollary~2]{lindenstr64}. In distinction from that result, the imbedding into $L_1(0,1)$ given by formulas \eqref{eq:imbed}--\eqref{eq:xi,eta} is quite explicit. Another difference is that our method is elementary and the proof is self-contained. \big(Also, formula \eqref{eq:N=} is simpler than, and therefore in some situations may be preferable to, \eqref{eq:imbed}--\eqref{eq:xi,eta}.\big) On the other hand, the study \cite{lindenstr64} contains a number of results that are more general than Corollary~2 therein. 

It is well known that any Euclidean space 
is linearly isometric to a subspace of $L_1$. 
Indeed,  
for all $x\in\R^d$ 
\begin{equation}\label{eq:eucl}
	\|x\|_2:=\sqrt{x\cdot x}=\sqrt{\frac\pi2}\,\int_{\R^d}|x\cdot t|\,\ga_d(\dd t), 
\end{equation}
where $\ga_d$ is the standard Gaussian measure on $\R^d$ and $\cdot$ denotes the standard inner product on $\R^d$. In place of $\ga_d$, one can similarly use any other spherically invariant measure $\nu$ on $\R^d$ such that 
$\int_{\R^d}|x\cdot t|\,\nu(\dd t)\in(0,\infty)$ for some or, equivalently, any nonzero vector $x$ in $\R^d$. 

Using the imbedding-into-$L_1$ formulas \eqref{eq:eucl} and \eqref{eq:imbed}, it is straightforward to verify Hlawka's inequality 
\begin{equation*}
	\|x+y+z\|+\|x\|+\|y\|+\|z\|\ge\|x+y\|+\|y+z\|+\|z+x\|
\end{equation*}
for all $x,y,z$ in $V$ when either the norm $\|\cdot\|$ on $V$ is Euclidean or $V$ is two-dimensional. 
Another way to show that Hlawka's inequality holds for any two-dimensional normed space was presented in \cite{kelly-etal65}. 

\begin{remark}\label{rem:}
In general, a normed space $V$ of any given dimension greater than $2$ is not linearly isometric to a subspace of $L_1$. Indeed, otherwise Hlawka's inequality would hold for all $x,y,z$ in $\R^3$. However, as pointed out e.g.\ in \cite{fechner14}, Hlawka's inequality fails to hold for some $x,y,z$ in $\R^3$ endowed with the supremum norm. 
\end{remark}

\section{Applications}
First here, one has the following representation of the expected norm of a random vector $X$ in $V$ in terms of the characteristic function (c.f.) $V^*\ni\ell\mapsto\E e^{iX\ell}$ of $X$. 
\begin{corollary}
If the additive decomposition \eqref{eq:general V} holds and $X$ is any random vector in $V$, then 
\begin{equation}\label{eq:pos}
\E\|X\|=\frac2\pi\,\int_{V^*\times(0,\infty)}\frac{1-\Re\E e^{itX\ell}}{t^2}\,\mu(\dd\ell)\dd t. 	
\end{equation}
\end{corollary}
This follows immediately from (say) \cite[Corollary~2]{positive}. 
A similar representation of $\E\|X\|$ in terms of the Fourier--Laplace transform of the distribution of $X$ can be just as easily obtained based on \cite[Theorem~1]{positive}. 
In view of \eqref{eq:N=} and \eqref{eq:eucl}, these representations of $\E\|X\|$ are quite explicit if $V$ two-dimensional or Euclidean.  

Removing both instances of the expectation from \eqref{eq:pos} (that is, replacing $X$ there with a nonrandom vector $x\in V$), then raising both sides to the $j$th power, and finally reapplying the expectation, one obtains  
\begin{equation}\label{eq:pos,power}
\E\|X\|^j=\Big(\frac2\pi\Big)^j\,\int_{\big(V^*\times(0,\infty)\big)^j}
\bigg(\E\prod_{\al=1}^j\big(1-\Re e^{it_\al X\ell_\al}\big)\bigg)\,\prod_{\al=1}^j\frac{\mu(\dd\ell_\al)\dd t_\al}{t_\al^2} 	
\end{equation}
for any natural $j$. 
Note that the expectation in \eqref{eq:pos,power} can be easily expressed in terms of the c.f.\ of $X$, namely, as 
\begin{equation}
	\sum_{(A,B)}\Big(-\frac12\Big)^{|A\cup B|}\,
	\E\exp\Big\{iX\Big(\sum_{\al\in A}t_\al\ell_\al-\sum_{\beta\in B}t_\beta\ell_\beta\Big)\Big\}, 
\end{equation}
where $\sum_{(A,B)}$ denotes the summation over all ordered pairs $(A,B)$ of disjoint subsets of the set $\{1,\dots,j\}$ and $|A\cup B|$ denotes the cardinality of the set $A\cup B$. 

Because of the multiplicativity property of the c.f.\ with respect to the convolution, representations such as \eqref{eq:pos} and \eqref{eq:pos,power} may be especially useful when $X$ is the sum of independent random vectors. 
\begin{center}
	***
\end{center}

Suppose now that $x_1,\dots,x_n$ are any vectors in any normed space $V$ and $\vp_1,\dots,\vp_n$ are independent Rademacher r.v.'s, so that $\P(\vp_i=1)=\P(\vp_i=-1)=\frac12$ for all $i$. 
Lata{\l}a and Oleszkiewicz \cite{latala-olesz_khin-kahane} made an elementary but very ingenious argument to show that 
\begin{equation}\label{eq:lat-olesz}
	2\Big(\E\Big\|\sum \vp_i x_i\Big\|\Big)^2\ge\E\Big\|\sum \vp_i x_i\Big\|^2. 
\end{equation}
Previously, Szarek \cite{szarek76} obtained this result in the case $V=\R$, which had been a long-standing conjecture of Littlewood; see e.g.\ \cite{hall.r75}.  
The constant factor $2$ in \eqref{eq:lat-olesz} is 
the best possible, even for $V=\R$ (take $n=2$ and $x_1=x_2\ne0$). 

In the case when the norm admits an additive decomposition, the lower bound $\E\Big\|\sum \vp_i x_i\Big\|^2$ in \eqref{eq:lat-olesz} can be improved: 

\begin{corollary}
If decomposition \eqref{eq:general V} holds, then     
\begin{equation}\label{eq:cor1}
	2\Big(\E\Big\|\sum \vp_i x_i\Big\|\Big)^2\ge\Big(\int_{V^*}\sqrt{\sum(x_i\ell)^2}\,\mu(\dd\ell)\Big)^2.  
\end{equation}
\end{corollary}
This follows immediately from \eqref{eq:general V} and Szarek's result. Moreover, if \eqref{eq:general V} holds, then one can rewrite the lower bounds in \eqref{eq:lat-olesz} and \eqref{eq:cor1} respectively as 
$$
\text{$\int\limits_{V^*\times V^*}\!\!\!\E\big|\tsum\vp_i x_i\ell\big|\,\big|\tsum\vp_j x_j m\big|\,\mu(\dd\ell)\mu(\dd m)$ and \hspace*{-10pt}
$\int\limits_{V^*\times V^*}\!\!\!\!\sqrt{\tsum(x_i\ell)^2\tsum(x_j m)^2}\,\mu(\dd\ell)\mu(\dd m)$.}
$$
So, by the Cauchy--Schwarz inequality, the lower bound in \eqref{eq:cor1} is no less than that in \eqref{eq:lat-olesz}. Moreover, the bound in \eqref{eq:cor1} may be of simpler structure and easier to compute (without having to use the expectation), especially in the two-dimensional case, when one has the explicit additive decomposition \eqref{eq:N=} of the norm. 
On the other hand, \eqref{eq:lat-olesz} holds for any norm and any dimension. 

\begin{center}
	***
\end{center}

In conclusion, consider the inequality 
\begin{equation}\label{eq:buja}
	\E\|X-Y\|_2\le\E\|X+Y\|_2, 
\end{equation}
where $X$ and $Y$ are independent identically distributed (iid) random vectors in $\R^d$ and $\|\cdot\|_2$ is the Euclidean norm, as in \eqref{eq:eucl}. This inequality was obtained in \cite{buja-etal94}. 
As noted in \cite{lifsh-tyurin}, in the case $d=1$ \eqref{eq:buja} follows immediately from the identity 
\begin{equation*}
	\E|X+Y|=\E|X-Y|+2\int_0^\infty[\P(X>r)-\P(X<-r)]^2\,\dd r. 	
\end{equation*}
Now the $L_1$-imbedding formula \eqref{eq:imbed} immediately yields 

\begin{corollary}
For any two-dimensional normed space $V$ and any iid random vectors $X$ and $Y$ in $V$,  
\begin{equation}\label{eq:cor}
	\E\|X-Y\|\le\E\|X+Y\|.  
\end{equation}
\end{corollary}

As shown by Johnson \cite{2iidVectors}, for each natural $d\ge3$ inequality \eqref{eq:cor} fails to hold for $V=\R^d$ in general. Indeed, define the norm on $\R^d$ by the formula 
\begin{equation*}
	\|x\|:=\max\{|x_i|\vee|x_i-x_j|\colon i,j=1,\dots,d\} 
\end{equation*}
for $x=(x_1,\dots,x_d)\in\R^d$. 
Let $(e_1,\dots,e_d)$ be the standard basis of $\R^d$. Let $X$ and $Y$ be iid random vectors in $\R^d$. 
For any natural $d\ge4$, suppose that the random vector $X$ is such that $\P(X=e_i)=\frac1d$ for each $i=1,\dots,d$. Then $\E\|X-Y\|=2\frac{d-1}d>\frac{d+1}d=\E\|X+Y\|$. 
In the remaining case when $d=3$, suppose that the random vector $X$ is such that $\P(X=e_1)=\P(X=e_2)=\P(X=e_3)=\P\big(X=-\frac12\,(e_1+e_2+e_3)\big)=\frac14$. Then $\E\|X-Y\|=\frac{21}{16}>\frac{19}{16}=\E\|X+Y\|$. 


{\bf Acknowledgment.}\ This note was sparked by answers by Noam D.\ Elkies and Suvrit Sra on MathOverflow \cite{MO_hlawka} and 
William B. Johnson's comments there. 

\bibliographystyle{abbrv}
\bibliography{C:/Users/ipinelis/Dropbox/mtu/bib_files/citations12.13.12}

\def\cprime{$'$} \def\polhk#1{\setbox0=\hbox{#1}{\ooalign{\hidewidth
  \lower1.5ex\hbox{`}\hidewidth\crcr\unhbox0}}}
  \def\polhk#1{\setbox0=\hbox{#1}{\ooalign{\hidewidth
  \lower1.5ex\hbox{`}\hidewidth\crcr\unhbox0}}}
  \def\polhk#1{\setbox0=\hbox{#1}{\ooalign{\hidewidth
  \lower1.5ex\hbox{`}\hidewidth\crcr\unhbox0}}} \def\cprime{$'$}
  \def\polhk#1{\setbox0=\hbox{#1}{\ooalign{\hidewidth
  \lower1.5ex\hbox{`}\hidewidth\crcr\unhbox0}}}
  \def\polhk#1{\setbox0=\hbox{#1}{\ooalign{\hidewidth
  \lower1.5ex\hbox{`}\hidewidth\crcr\unhbox0}}} \def\cprime{$'$}
  \def\cprime{$'$}
\begin{thebibliography}{10}

\bibitem{MO_hlawka}
Absolute value inequality for complex numbers, 2015.
\newblock MathOverflow,
  \url{http://mathoverflow.net/questions/167685/absolute-value-inequality-for-complex-numbers}.

\bibitem{2iidVectors}
An inequality for two independent identically distributed random vectors in a
  normed space, 2015.
\newblock MathOverflow,
  \url{http://mathoverflow.net/questions/208194/an-inequality-for-two-independent-identically-distributed-random-vectors-in-a-no/208245#208250}.

\bibitem{buja-etal94}
A.~Buja, B.~F. Logan, J.~A. Reeds, and L.~A. Shepp.
\newblock Inequalities and positive-definite functions arising from a problem
  in multidimensional scaling.
\newblock {\em Ann. Statist.}, 22(1):406--438, 1994.

\bibitem{fechner14}
W.~Fechner.
\newblock Hlawka's functional inequality.
\newblock {\em Aequationes Math.}, 87(1-2):71--87, 2014.

\bibitem{hall.r75}
R.~R. Hall.
\newblock On a conjecture of {L}ittlewood.
\newblock {\em Math. Proc. Cambridge Philos. Soc.}, 78(3):443--445, 1975.

\bibitem{kelly-etal65}
L.~M. Kelly, D.~M. Smiley, and M.~F. Smiley.
\newblock Two dimensional spaces are quadrilateral spaces.
\newblock {\em Amer. Math. Monthly}, 72:753--754, 1965.

\bibitem{latala-olesz_khin-kahane}
R.~Lata{\l}a and K.~Oleszkiewicz.
\newblock On the best constant in the {K}hinchin-{K}ahane inequality.
\newblock {\em Studia Math.}, 109(1):101--104, 1994.

\bibitem{lifsh-tyurin}
M.~Lifshits, R.~M. Schilling, and I.~Tyurin.
\newblock A probabilistic inequality related to negative definite functions.
\newblock In {\em High dimensional probability VI}, volume~66 of {\em Progress
  in Probability}, pages 73--80. Springer, Basel, 2013.

\bibitem{lindenstr64}
J.~Lindenstrauss.
\newblock On the extension of operators with a finite-dimensional range.
\newblock {\em Illinois J. Math.}, 8:488--499, 1964.

\bibitem{positive}
I.~Pinelis.
\newblock Positive-part moments via the {F}ourier--{L}aplace transform.
\newblock {\em J. Theor. Probab.}, 24:409--421, 2011.

\bibitem{szarek76}
S.~J. Szarek.
\newblock On the best constants in the {K}hinchin inequality.
\newblock {\em Studia Math.}, 58(2):197--208, 1976.

\end{thebibliography}


\end{document}